\documentclass[11pt,a4paper]{article}
\usepackage[OT1]{fontenc}
\usepackage[english]{babel}
\usepackage{amsfonts}
\usepackage{amsthm}
\def\h{ {\cal H} }
\def\l{ {\cal L} }
\def\a{ {\cal A} }

\def\b{ {\cal B} }
\def\u{ {\cal U} }
\def\m{ {\cal M} }
\def\n{ {\cal N} }

\def\ii{ {\cal I} }
\def\t{ {\cal T} }
\def\s{ {\cal S} }

\def\p{ {\cal P} }

\def\xx{ {\bf x} }

\def\zz{ {\bf z} }

\def\tt{ {\bf t} }

\newtheorem{teo}{Theorem}[section]
\newtheorem{prop}[teo]{Proposition}
\newtheorem{lem}[teo]{Lemma}
\newtheorem{coro}[teo]{Corollary}

\theoremstyle{definition}
\newtheorem{rem}[teo]{Remark}

\title{Geodesics of projections in von Neumann algebras}
\author{Esteban Andruchow\footnote{{\sc  {Instituto Argentino de Matem\'atica, `Alberto P. Calder\'on', CONICET, Saavedra 15 3er. piso,
(1083) Buenos Aires, Argentina;
and Universidad Nacional de General Sarmiento, J.M. Gutierrez 1150 (1613), Los Polvorines, Argentina
}} e-mail: eandruch@ungs.edu.ar}}

\begin{document}

\maketitle 

\begin{abstract}
Let $\a$ be a von Neumann algebra and $\p_\a$ the manifold of projections in $\a$. There is a natural linear connection in $\p_\a$, which in the finite dimensional case coincides with the  the Levi-Civita connection of the Grassmann manifold of $\mathbb{C}^n$. In this paper we show that two projections $p,q$ can be joined by a geodesic, which  has minimal length (with respect to the metric given by the usual norm of $\a$), if and only if
$$
p\wedge q^\perp\sim p^\perp\wedge q,
$$
where $\sim$ stands for the Murray-von Neumann equivalence of projections. It is shown that the minimal geodesic is unique if and only if $p\wedge q^\perp= p^\perp\wedge q=0$. 
If $\a$ is a finite factor, any pair of projections in the same connected component of $\p_\a$ (i.e., with the same trace) can be joined by a minimal geodesic.

We explore certain relations with Jones' index theory for subfactors. For instance, it is shown that if $\n\subset\m$ are {\bf II}$_1$ factors with finite index $[\m:\n]=\tt^{-1}$, then the geodesic distance $d(e_\n,e_\m)$ between the induced projections $e_\n$ and $e_\m$ is $d(e_\n,e_\m)=\arccos(\tt^{1/2})$.
\end{abstract}

\bigskip



\section{Introduction}
If $\a$ is a C$^*$-algebra, let $\p_\a$ denote the set of (selfadjoint) projections in $\a$. $\p_\a$ has a rich geometric structure, see for instante the papers \cite{pr} by H. Porta and L.Recht and \cite{cpr} by G. Corach, H. Porta and L. Recht. In these works, it was shown that $\p_\a$ is a C$^\infty$ complemented submanifold of $\a_{s}$, the set of selfadjoint elements of $\a$, and has a natural linear connection, whose geodesics can be explicitly computed. A  metric is introduced, called in this context a Finsler metric: since the tangent spaces of $\p_\a$ are closed and complemented  linear subspaces of $\a_s$, they can be endowed with the norm metric. With this Finsler metric, Porta and Recht \cite{pr} showed that two projections $p,q\in\p_\a$ which satisfy that $\|p-q\|<1$ can be joined by a unique geodesic, which is minimal for the metric (i.e., it is shorter than any other smooth curve in $\p_\a$ joining the same endpoints).

In general, two projections $p,q$ in $\a$ satisfy that $\|p-q\|\le 1$, so that what  remains to consider is what happens in the extremal case $\|p-q\|=1$: under what conditions does there exist a geodesic, or a minimal geodesic, joining them. 

In the case when $\a=\b(\h)$ the algebra of all bounded linear operators in a Hilbert space $\h$, it is known (see for instance \cite{jmaa}) that there exists a geodesic joining $p$ and $q$ if and only if
$$
\dim R(p)\cap N(q)=\dim N(p)\cap R(q).
$$
The geodesic is unique if and only if these intersections are trivial. 

The purpose of this note is to show that these facts remain valid if $\a$ is a von Neumann algebra, if we replace $\dim$ by the dimension relative to $\a$. Namely, it is shown that there exists a minimal geodesic joining $p$ and $q$ in $\p_\a$, if and only if
$$
p\wedge q^\perp \sim p^\perp \wedge q.
$$
Here $\wedge$ denotes the infimum of two projections, $p^\perp=1-p$, and $\sim$ is the Murray-von Neumann equivalence of projections. Also, it is shown that there exists a unique minimal geodesic if and only if
$$
p\wedge q^\perp=0=p^\perp \wedge q.
$$
We show that if $\a$ is a finite factor, any pair of projections in $\a$ in the same connected component of $\p_\a$ (i.e., with the same trace), can be joined by a minimal geodesic.

In the final section of this paper, we explore the relationship with the index theory of von Neumann factors, introduced by V. Jones in \cite{jones}. A pairing $\n\subset \m$ of factors of type {\bf II}$_1$, induces a sequence of projections, by means of the {\it basic construction}. We show that one recovers Jones index as a geodesic distance (minima of lengths of curves joining two given points): if $e, f$ are two consecutive terms in the sequence of projections, then 
$$
d(e,f)=\arccos(\tt^{1/2}),
$$
where $\tt^{-1}=[\m:\n]$. Also we show that if $\n_0,\n_1\subset\m$ with Jones' projections $e_0, e_1$, satisfy that $\|e_0-e_1\|<1$, then the unique geodesic $\delta(t)$ induces a smooth path of  conditional expectations between $\m$ and intermediate factors $\n_t$, and the parallel transport of this geodesic, induces a smooth path of  normal $*$-isomorphimsms between $\n_0$  and $\n_t$.

\section{Preliminaries}

The space $\p_\a$ is sometimes called the Grassmann manifold of $\a$. The reason for this name is that in the case when $\a=\b(\h)$, $\p_{\b(\h)}$ parametrizes the set of closed subspaces of $\h$: to each closed subspace $\s\subset\h$ corresponds the orthogonal projection $P_\s$ onto $\s$. Let us describe below the main features of the geometry of $\p_\a$ in the general case.

\subsection{Homogeneous structure}

Denote by $\u_\a=\{u\in\a: u^*u=uu^*=1\}$
 the unitary group of $\a$. It is a Banach-Lie group, whose Banach-Lie algebra is
$\a_{as}=\{x\in\a: x^*=-x\}$. This group acts on $\p_\a$ by means of $u\cdot p=upu^*, \ u\in\u_\a, \ p\in\p_\a$.
The action is smooth and locally transitive. It is known (see \cite{pr}, \cite{cpr}) that $\p_\a$ is what in differential geometry is called a {homogeneous space} of the group $\u_\a$. The local structure of $\p_\a$ is described using this action. For instance, the tangent space $(T\p_\a)_p$ of $\p_\a$ at $p$ is given by  $(T\p_\a)_p=\{x\in\a_s: x=px+xp\}$.

The isotropy subgroup of the action at $p$, i.e., the elements of $\u_\a$ which fix a given $p$, is $\ii_p=\{v\in\u_\a: vp=pv\}$.
The isotropy algebra $\mathfrak{I}_p$ at $p$ is its Banach-Lie algebra
$\mathfrak{I}_p=\{y\in\a_{as}: yp=py\}$.

It is useful, in order to describe and understand the geometry of $\p_\a$, to consider the {\it diagonal / co-diagonal} decomposition of $\a$ in terms of a fixed projection $p_0\in\p_\a$. Elements $x\in\a $ which commute with $p_0$, or equivalently, commute with the symmetry $2p_0-1$, when written as $2\times 2$ in terms of $p_0$, have diagonal matrices.  Co-diagonal matrices correspond with elements in $\a$ which anti-commute with $2p_0-1$.

Then, the isotropy subgroup and the isotropy algebra $\ii_{p_0}, \mathfrak{I}_{p_0}$ at $p_0$, are respectively the sets of diagonal unitaries and diagonal anti-Hermitian elements of $\a$. On the other side, the tangent space $(T\p_\a)_{p_0}$ is the set of diagonal selfadjoint  elements of $\a$.

\subsection{Reductive structure}

Given an homogeneous space, a {\it reductive } structure is a smooth distribution $p\mapsto  {\bf H}_p\subset\a_{as}$, $p\in\p_\a$, of supplements of $\mathfrak{I}_p$ in $\a_{as}$, which is invariant under the action of $\ii_p$. That is, a distribution ${\bf H}_p$ of closed linear subspaces of $\a_{as}$ verifying that ${\bf H}_p\oplus \mathfrak{I}_p=\a_{as}$; $v{\bf H}_pv^*={\bf H}_p$ for all $v\in\ii_p$;  and the map $p\mapsto {\bf H}_p$ is smooth.

In the case  of $\p_\a$, the choice of the (so called) {\it horizontal}  subspaces ${\bf H}_p$ is  natural. The horizontal ${\bf H}_p$  defined in \cite{cpr} is  
${\bf H}_p=\{\left(\begin{array}{cc} 0 & z \\ -z^* & 0 \end{array} \right): z\in p\a p^\perp\}$, i.e., the set of co-diagonal anti-Hermitian elements of $\a$

As in classical differential geometry, a reductive structure on a homogeneous space defines a linear connection: if $X(t)$ is a smooth curve of vectors tangent to a smooth curve $p(t)$ in $\p_\a$, i.e., a smooth curve of selfadjoint elements of $\a$, which are pointwise co-diagonal with respect to $p(t)$, then the covariant derivative of the linear connection is given by
$$
\frac{D}{dt}X(t):=\hbox{diagonal part w.r.t. } p(t) \hbox{ of } \dot{X}(t)=p(t)\dot{X}(t)p(t)+p^\perp(t)\dot{X}(t)p^\perp(t).
$$
It is not difficult to deduce then that a {\it geodesic}  starting at $p_0\in\p_\a$ is given by the action of a one parameter group with horizontal (anti-Hermitian co-diagonal) velocity on $p_0$ . Namely, given the base point $p_0\in\p_\a$, and a tangent vector ${\bf x}=\left(\begin{array}{cc} 0 & x \\ x^*  & 0 \end{array} \right)\in (T\p_\a)_{p_0}$, the unique geodesic $\delta$ of $\p_\a$ with $\delta(0)=p_0$ and $\dot{\delta}(0)={\bf x}$ is given by
$$
\delta(t)= e^{tz_\xx}p_0e^{-tz_\xx},
$$
where $z_{\xx}:=\left(\begin{array}{cc} 0 & -x \\ x^*  & 0 \end{array} \right) $.
The horizontal element $z_\xx$ is characterized as the unique horizontal element at $p_0$ such that $[z_\xx,p_0]=\xx$.

\subsection{Finsler metric}

As we mentioned above, one endows each tangent space $(T\p_\a)_p$ with the usual norm of $\a$. We emphasize that this (constant) distribution of norms is not a Riemannian metric (the C$^*$-norm is not given by an inner product), neither is it a Finsler metric in the classical sense (the map $a\mapsto \|a\|$ is non differentiable). Therefore the minimality result which we describe below does not follow from general considerations. It was proved in \cite{pr} using ad-hoc techniques. 
\begin{enumerate}
\item
Given $p\in\p_\a$ and $\xx\in(T\p_\a)_p$, normalized so that $\|\xx\|\le\pi/2$, then the geodesic $\delta$ remains minimal for all $t$ such that $|t|\le 1$.
\item
Given $p,q\in\p_\a$ such that $\|p-q\|<1$, there exists a unique minimal geodesic $\delta$ such that $\delta(0)=p$ and $\delta(1)=q$.
\end{enumerate}
We shall call these geodesics (with initial speed $\|\xx\|\le \pi/2$) {\it normalized} geodesics.

\section{Von Neumann algebras}

In this paper we consider the case when $\a$ is a von Neumann algebra. We shall suppose $\a$ acting in a Hilbert space $\h$ (i.e., $\a\subset\b(\h)$). As we shall see, this representation is auxiliary, and the results on the geometry of $\p_\a$ do not depend on the representation. The main assertion of this section is that the conditions of existence and uniqueness of minimal geodesics joining given projections $p,q\in\p_\a$ are the a natural generalization of the conditions valid in the case of $\b(\h)$. 

If $p,q\in\p_\a$, we denote by $p^\perp=1-p$, and by $p\wedge q$ the projection onto $R(p)\cap R(q)$ (which belongs to $\p_\a$); $p$ and $q$ are said to be {\it Murray\ -\ von Neumann equivalent}, in symbols $p\sim q$, if there exists $v\in\a$ (a partial isometry) such that $v^*v=p$ and $vv^*=q$. 
Our main result follows:
\begin{teo}\label{teorema1}
Let $p,q\in\p_\a$.
\begin{enumerate}
\item
There exists a geodesic $\delta$  of $\p_\a$ joining $p$ and $q$ if and only if
$$
p\wedge q^\perp\sim p^\perp\wedge q.
$$
Moreover, the geodesic can be chosen minimal (i.e., normalized).
\item
There is a unique normalized geodesic if and only if $p\wedge q^\perp=p^\perp\wedge q =0$.
\end{enumerate}
\end{teo}
\begin{proof}
Existence: suppose first that $p\wedge q^\perp\sim p^\perp\wedge q$. 
Consider following projections which sum $1$ and commute both with $p$ and $q$:
$$
e_{11}=p\wedge q \ , \ e_{00}=p^\perp\wedge q^\perp \ , \ e_{10}=p\wedge q^\perp \ , \ e_{11}=p^\perp\wedge q \ , \ e_0=1-\sum_{i,j=0,1} e_{i,j}.
$$
It is straightforward to verify that $e_{ij}$ commute with $p$ and $q$, and thus $e_0$ also does. The decomposition of the Hilbert space induced by these projections is sometimes called the Halmos decomposition of the space, in the presence of two closed subspaces ($R(p)$ and $R(q)$); the last subspace $R(e_0)$, is called the generic part of $p$ and $q$.
We shall construct the exponent $\xx$ of the geodesic joining $p$ and $q$ as a sum of anti-Hermitian elements in $\a$,
$$
\xx=\xx'+\xx''+\xx_0,
$$
where  $\xx'$ acts in the range of $e_{11}+e_{00}$, $\xx''$ acts in the range of $e_{10}+e_{01}$ and $\xx_0$  acts in the range of $e_0$. Moreover, each of these elements is co-diagonal with respect to the corresponding reduction of $p$ to these subspaces. First note that $pe_{ii}=qe_{ii}$ (on $e_{00}$ they are both zero, on $e_{11}$ they are both the identity). Thus the exponent $\xx'$ can be chosen $0$.

Let us consider next the part in $e_0$. Here we make use of the representation $\a\subset \b(\h)$. Denote by $\h_0=R(e_0)$, and by $p_0=pe_0$, $qe_0$ the reductions of $p,q$  to this subspace $\h_0$. Then, it is clear that $p_0, q_0$ lie in {\it generic position} (\cite{halmos}, \cite{dixmier}): their ranges and nullspaces intersect trivially. Thus, by a result by P. Halmos \cite{halmos}, there exist a Hilbert space $\l$, a positive operator $X\in\b(\l)$ ($\|X\|\le \p/2$ and a unitary isomorphism $\h_0\to\l\times\l$ which carries 
$$
p_0 \ \hbox{ to } \ \left(\begin{array}{cc} 1 & 0 \\ 0 & 0 \end{array} \right)\ , \  \hbox{ and } q_0\  \hbox{ to }\  \left(\begin{array}{cc} \cos^2(X) & \cos(X)\sin(X) \\ \cos(X)\sin(X) & \sin^2(X) \end{array} \right).
$$
Between these operator matrices, one can find the  (co-diagonal) exponent 
$$
Z=\left(\begin{array}{cc} 0 & X \\ -X & 0 \end{array} \right),
$$
which satisfies 
$$
e^Z\left(\begin{array}{cc} 1 & 0 \\ 0 & 0 \end{array} \right)e^{-Z}=\left(\begin{array}{cc} \cos^2(X) & \cos(X)\sin(X) \\ \cos(X)\sin(X) & \sin^2(X) \end{array} \right).
$$
These are straightforwward verifications, and provide the exponent for a geodesic joining the two operator matrices. One loses track though  of how elements of $\a$ are changed by the Halmos isomorphism. The key fact to relate these matrices to the former projections $p_0,q_0$ is the following elementary identity proved in \cite{jmaa}
\begin{equation}\label{formula}
e^Z=\left(\begin{array}{cc} 1 & 0 \\ 0 & -1 \end{array} \right) V,
\end{equation}
where $V$ is the unitary part in the polar decomposition of 
$$
B-1=\left(\begin{array}{cc} 1 & 0 \\ 0 & 0 \end{array} \right)+\left(\begin{array}{cc} \cos^2(X) & \cos(X)\sin(X) \\ \cos(X)\sin(X) & \sin^2(X) \end{array} \right)-1,
$$
i.e. $B-1=V|B-1|$. Again, this is an elementary matrix computation. Let $b_0=p_0+q_0$, and let $v_0$ be the isometric part in the polar decomposition (recall that $e_0$ is the unit in this part of the algebra)
$$
b_0-e_0=v_0|b_0-e_0|.
$$
Clearly $v_0\in\a$ and is carried by the Halmos isomorphism to $V$. Therefore, if one regards (\ref{formula}), it follows that the unitary element $v_0(2p_0-1)$ is carried by this ismorphism to $e^Z$, i.e. $e^Z$ corresponds to an element of $\a$. Moreover,
$$
\|Z\|=\|X\|\le \pi/2,
$$
which implies that $Z$ is the unique anti-Hermitian logarithm of $e^Z$ with spectrum in $(-i\pi,i\pi)$. It follows that there exists a unique element $\xx_0\in\a$ which corresponds to $Z$, and therefore satisfies 
$$
e^{\xx_0}p_0e^{-\xx_0}=q_0.
$$
It remains to construct the exponent $\xx''$ acting in $e_{10}+e_{01}$. Note that the reductions of $p$ and $q$ to this part are 
$$
p(e_{10}+e_{01})=p(p\wedge q^\perp+p^\perp \wedge q)=p\wedge q^\perp, 
$$
ans similarly $q(e_{10}+e_{01})=p^\perp \wedge q$. By hypothesis, there is a partial isometry $w\in\a$ such that 
$$
w^*w=p\wedge q^\perp \ \ \hbox{ and } \ \ ww^*=p^\perp\wedge q.
$$
Then $\xx''=i\frac{\pi}{2} (w+w^*)$ does the feat: since $p\wedge q^\perp \perp p^\perp \wedge q$, it follows that $\xx''$ is $p_{10}+p_{01}$ co-diagonal. Clearly $\|\xx''\|=\pi/2$. Note also that  $w^2=0$, so that 
$$
e^{\xx''}=i(w+w^*).
$$
Finally,
$$
e^{\xx''}(p\wedge q^\perp)=i(w+w^*)(w^*w)=iww^*w=iww^*(w+w^*)=(p^\perp\wedge q) e^{\xx''},
$$
i.e., $e^{\xx''}$ intertwines the reductions of $p$ and $q$ to this part.

If we put together $\xx=\xx'+\xx''+\xx_0$, which is an orthogonal sum, we have a $p$-co-diagonal anti-Hermitian element of $\a$, with $\|\xx\|\le\pi/2$ (note that $\xx''$ might be zero, if $p\wedge q^\perp=p^\perp\wedge q=0$), which satisfies
$$
e^{\xx}pe^{-\xx}=q.
$$

Conversely, suppose that there exists a normalized geodesic wich joins $p$ and $q$, i.e. there exists a $p$-co-diagonal anti-Hermitian element $\xx\in\a$ with $\|\xx\|\le \pi/2$ such that $e^{\xx}pe^{-\xx}=q$.  We claim that  $e^{\xx}$ maps $R(p\wedge q^\perp)$ onto $R(p^\perp\wedge q)$. Clearly $e^{\xx}$ maps $R(p)$ onto $R(q)$. Pick $\xi\in R(p\wedge q^\perp)=R(p)\cap N(q)$.  Then $e^{\xx}\xi\in R(q)$. It was noted in \cite{pr}, that the fact that $\xx$ is $p$-co-diagonal means that $\xx$  anti-commutes with $2p-1$. Thus,  
$$
(2p-1)e^{\xx}=e^{-\xx}(2p-1).
$$
Then, since $(2p-1)\xi=\xi$ and $(2q-1)\xi=-\xi$, 
$$
(2p-1)e^{\xx}\xi=e^{-\xx}(2p-1)\xi=e^{-\xx}\xi=-e^{-\xx}(2q-1)\xi=-(2p-1)e^{-\xx}\xi=-e^{\xx}(2p-1)\xi=-e^{\xx}\xi,
$$
i.e., $e^{\xx}\xi\in N(p)$, and thus  $e^{\xx}(R(p)\cap N(q))\subset R(q)\cap N(p)$
The other inclusion follows similarly (or by symmetry: in fact $\xx$ is also $q$-co-diagonal, because $-\xx$ is  the initial velocity of the reversed geodesic which starts at $q$). It folllows that $w=e^{\xx}(p\wedge q^\perp)\in\a$ is a partial isometry with initial space $p\wedge q^\perp$ and final space $p^\perp\wedge q$. 

Uniqueness: if $p\wedge q^\perp=p^\perp\wedge q=0$, then $R(p)\cap N(q)=N(p)\cap R(q)=\{0\}$, and there exists a unique normalized geodesic in $\p_{\b(\h)}$ joining $p$ and $q$. By the first part of the proof, there is a normalized geodesic joining them in $\p_\a$. Thus, it is unique. 

Conversely, suppose that there exists a unique geodesic joining $p$ and $q$. Then necessarily $p\wedge q^\perp\sim p^\perp\wedge q$. Suppose that these projections are non zero. Then, there are infinitely many different partial isometries $w$ such that 
$w^*w=p\wedge q^\perp$ and $ww^*=p^\perp\wedge q$. As in the first part of the proof, any such $w$ give rise to different exponents $\xx''$, and thus different $\xx$, i.e. different geodesics joining $p$ and $q$.  
\end{proof}

\begin{rem}
In the above result, it was shown in fact that the submanifold $\p_\a\subset \p_{\b(\h)}$ is totally geodesic: the geodesics of $\p_\a$ are geodesics of the bigger manifold $\p_{\b(\h)}$; if $p,q\in\p_\a$ are joined by a unique geodesic of $\p_{\b(\h)}$, then  this geodesic remains inside $\p_\a$.
\end{rem}

\section{Hopf-Rinow theorem in finite factors}

Two subspaces of dimension $k$ in $\mathbb{C}^n$ can be joined by a minimal geodesic of the Levi-Civita connection in the Grassmann manifold. This fact can be proved using the projection formalism. That is,  parametrizing subspaces with othogonal projections in $M_n(\mathbb{C})$, by means of 
$$
\mathbb{C}^n\supset \s\longleftrightarrow P_\s \in M_n(\mathbb{C}),
$$
where $P_\s$ is the orthogonal projection onto $\s$.
Two subspaces $\s,\t\subset\mathbb{C}^n$ have the same dimension if and only if the corresponding projections $P_\s,P_\t$ have the same rank, i.e.
$$
Tr(P_\s-P_\t)=0.
$$
Let us see that in this case one has, automatically, that 
$$
\dim(\s\cap\t^\perp)=\dim(\s^\perp\cap\t).
$$
This fact has an elementary proof. Let us prove it in a non totally elementary fashion, which will allow us to obtain a generalization. The operator $A=P_\s-P_\t$ is a selfadjoint contraction, and if $B=P_\s+P_\t$, 
$$
N(B-1)=\s\cap\t^\perp \ \oplus \ \s^\perp\cap\t=N(A-1)\oplus N(A+1).
$$
On the subspace $N(B-1)^\perp$, $B-1$ is an invertible matrix, and the symmetry $V$ in its polar decomposition
$B-1=V|B-1|$ satisfies that $VP_\s V=P_\t$ in $N(B-1)^\perp$.  Then $A$ is reduced by $N(B-1)$, and
$$
V (A\Big|_{N(B-1)^\perp})V=-A\Big|_{N(B-1)^\perp}.
$$
This implies that the spectrum of $A\Big|_{N(B-1)^\perp}$ is symmetric with respect to the origin: if $\lambda$ is an eigenvalue of $A$ with $|\lambda|<1$, then $-\lambda$ is also an eignevalue of $A$, and they have the same multiplicity: $\dim(N(A-\lambda))=\dim(N(A+\lambda))$. Then
$$
A=-P_{N(A+1)}+P_{N(A-1)}+A\Big|_{N(B-1)^\perp}=-P_{N(A+1)}+P_{N(A-1)}+\sum_{0<\lambda<1} \lambda (P_{N(A-\lambda)}-P_{N(A+\lambda)}).
$$
Thus, the fact that $Tr(A)=0$, means that 
$
Tr(P_{N(A-1)})=Tr(P_{N(A+1)}).
$
Thus $P_\s$ and $P_\t$ can be joined by a normalized geodesic. Remarkably, this geodesic is minimal for the Levi-Civita connection of the Grassmann manifold, but also, using the projection formalism, for the operator norm of $M_n(\mathbb{C})$, the $p$-Schatten norms (\cite{trans}), or more generally, for  unitary invariant norms (see \cite{alv}).

Let us suppose now that $\a$ is a finite von Neumann factor, with trace $\tau$. We shall see that the above argument holds (essentially unaltered):
\begin{teo}
Let $\a$ be a finite von Neumann factor with faithful normal trace $\tau$. Two projections $p,q\in\p_\a$  with $p\sim q$ (i.e., unitarily equivalent, or equivalently, in the same connected component of $\p_\a$) can be joined by a normalized geodesic.
\end{teo}
\begin{proof}
Let $a=p-q$, and again note that $N(a-1)=R(p)\cap N(q)$. Following previous notations, $P_{N(a-1)}=p\wedge q^\perp =e_{10}$. Similarly, $P_{N(a+1)}=p^\perp\wedge q=e_{01}$. Therefore,
$P_{N(b-1)}=e_{10}+e_{01}:=e''$. Again, since $e_{10}$ and $e_{01}$ are eigenspaces of $a$, these projections reduce $a$, let $a_0$ be the reduction of $a$ to $R(e'')^\perp$. Then, we have that
$$
a=e_{10}-e_{01}+a_0.
$$
 The operator $a_0$ is a difference of projections: $a_0=p_0-q_0$, where $p_0$ and $q_0$ are the reductions of $p$ and $q$ to $R(e'')^\perp$, with $N(a_0\pm e_0)=\{0\}$ ($e_0$ is the identity in $R(e'')^\perp$). It was shown by Chandler Davis \cite{davis} that there exists a symmetry $v_0$ ($v_0^*=v_0$, $v_0^2=e_0$; namely,
$v_0$ is the isometric part in the polar decomposition of $b_0-e_0$) such that
$$
v_0a_0v_0=-a_0.
$$
Let $\mu$ be the projection-valued spectral measure of $a_0$:
$$
a_0=\int_{-1}^1 \lambda d\mu(\lambda).
$$
As in the above argument in $M_n(\mathbb{C})$, the existence of the symmetry $v_0$ implies the symmetry of the spectral measure of $a_0$ with respect to the origin: if $\Lambda\subset [-1,1]$ is a Borel subset, then
$$
\mu(-\Lambda)=v_0\mu(\Lambda)v_0.
$$
Then 
$$
\tau(a_0)=\int_{-1}^1 \lambda d\tau(\mu(\lambda))=0,
$$
because the function  $f(\lambda)=\lambda$ is odd and the measure $\tau \mu$ is symmetric with respect to the origin:
$$
\tau(\mu(-\Lambda))=\tau(v_0\mu(\Lambda)v_0)=\tau(e_0\mu(\Lambda))=\tau(\mu(\Lambda)),
$$
because $\mu\le e_0$.
Then, since $p\sim q$, 
$$
0=\tau(p)-\tau(q)=\tau(a)=\tau(e_{10}-e_{01}+a_0)=\tau(e_{10})-\tau(e_{01}),
$$
so that  $\tau(p\wedge q^\perp)=\tau(p^\perp\wedge q)$, i.e., $p\wedge q^\perp\sim p^\perp\wedge q$. Therefore, by Theorem \ref{teorema1}, there exists a (minimal) normalized geodesic joining $p$ and $q$ in $\p_\a$.
\end{proof}
\begin{rem}
In \cite{int}, it was shown that in a finite algebra with faithful trace $\tau$, the geodesics have minimal length also when measured with the $\rho$ norms $\|\ \|\rho$ of the trace, for $\rho\ge 2$ ($\|x\|_\rho=(\tau(x^*x)^{\rho/2})^{1/\rho}$). Namely, it was shown that if $\delta(t)=e^{t\zz}pe^{-t\zz}$ is a normalized geodesic ($\|\zz\|\le\pi/2$)  with $\delta(1)=q$, and $\gamma$ is any other smooth curve in $\p_\a$ with $\gamma(t_0)=p$ and $\gamma(t_1)=q$, then
$$
\ell_\rho(\gamma):=\int_{t_0}^{t_1} \|\dot{\gamma}(t)\|_\rho dt \ge \ell_\rho(\delta)=\|\zz\|_\rho.
$$
\end{rem}
As we have seen, on finite factors, a version of the the Hopf-Rinow is valid in $\p_\a$, and the geodesics are minimal for the usual norm of $\a$ at every tangent space. 
However, as a consequence of the fact in the above remark, we have that for the $p$-norms in the tangent space, including the pseudo-Riemannian case $p=2$, there are no {\it normal} neighbourhoods if $\a$ is a type $II_1$ factor. Indeed, for $2\le\rho<\infty$, denote by 
$$
d_\rho(p,q)=\inf\{\ell_\rho(\gamma): \gamma \hbox{ is smooth and joins } p \hbox{ and } q \hbox{ in } \p_\a \}
$$
the metric induced in $\p_\a$ by the $\rho$-norm.
\begin{prop}
Let $\a$ a type $II_1$ factor and $2\le\rho<\infty$ Then there exist pairs of projections in $\p_\a$, which are arbitrarily close for the $d_\rho$ metric, which can be joined by infinitely many geodesics.
\end{prop}
\begin{proof}
Given $0<r\le\frac12$, let $p\in\p_\a$ such that $\tau(p)=r$. Let $q\in\p_\a$ such that $q\le p^\perp$ and $\tau(q)=r$ (consider the reduced factor $p^\perp\a p^\perp$, which is also of type $II_1$, and pick there a projection $q$ with (renormalized) trace   $\frac{r}{1-r}$). Then, the Halmos decomposition given by $p$ and $q$ yields (following the notation of the preceding section)
$$
e_{00}=0 , \ e_{11}=0 , \ e_{10}=p , \ e_{10}=p , \ e_{01}=q , \ e_0=0.
$$
Since $p\sim q$, there exist infinitely many $v\in\a$ such that $v^*v=p$ and $vv^*=q$. 
Any of these $v$ provides a geodesic joining $p$ and $q$, given by (see the last part of the proof of Theorem \ref{teorema1}) the exponent $\xx=i\frac{\pi}{2}(v+v^*)$.
The length of any of these geodesics is
$$
\|\xx\|_\rho=\tau((\xx^*\xx)^{\rho/2})^{1/\rho}=\frac{\pi}{2}\tau((v^*v+vv^*)^{\rho/2})^{1/\rho}=\frac{\pi}{2}2^{1/\rho}\tau(p)^{1/\rho}=\pi 2^{1/\rho-1} r^{1/\rho}.
$$
\end{proof}

\section{Applications to finite index subfactors}

V.F.R. Jones  introduced the  theory of index for subfactors of a  {\bf II}$_1$ factor in \cite{jones}. An inclusion $\n\subset\m$ of {\bf II}$_1$ factors is said to be of {\it finite index} if the relative dimension (or coupling constant) 
$$
[\m:\n]:=\dim_\n(L^2(\m,\tau))=\tt^{-1}
$$
is finite (see \cite{murray vn}). A sequence of projections arises in this circumstance, by means of Jones' {\it basic cosntruction}. Denote by $\tau$ the normalized trace of $\m$ (and of the subsequent finite extensions which will be considered). Let $e_\n$ be the orthogonal projection of $L^2(\m,\tau)$ onto $L^2(\n,\tau)$. This projection, restricted to $\m\subset L^2(\m,\tau)$, induces   the unique trace invariant conditional expectation $E_\n:\m\to\n$. Jones proved that the von Neumann algebra $\m_1=<\m,e_\n>$ generated in $\b(L^2(\m,\tau))$ by $\m$ and $e_\n$ is again a {\bf II}$_1$ factor, and that the inclusion $\m\subset\m_1$ has finite index, with $[\m_1:\m]=[\m:\n]$. Thus, iterating the basic construction,  a sequence of orthogonal projections arises:  $e_1=e_\n, e_2=e_\m, \dots$. We shall be concerned only with the first two. These projections recover the index:
$$
\tau(e_\n)=\tau(e_\m)=\tt=[\m:\n]^{-1}.
$$
In particular, they are unitarily equivalent in any factor of the tower of factors enabled by the basic construction, in which both lie. More precisely, Jones proved (\cite{jones}, Proposition 3.4.1) that
$$
e_\n e_\m e_\n=\tt e_\n \ \hbox{ and } \ e_\n^\perp\wedge e_\m=\ e_\n\wedge e_\m^\perp=0.
$$
It follows that $e_\n$ and $e_\m$ can be joined by a unique geodesic which lies in $\m_2=<\m, e_\n, e_\m>$. Thus, the finite index inclusion $\n\subset\m$ gives rise to a unique element $\zz_{\m,\n}\in\m_2$, the exponent of this geodesic:
$$
\zz_{\m,\n}^*=-\zz_{\m,\n} \ , \ d(e_\n,e_\m)=\|\zz_{\m,\n}\|\le\pi/2 \ , \zz_{\m,\n} \hbox{ is } e_\n \hbox{ and } e_\m \hbox{ co-diagonal,}
$$
and
$$
 e^{\zz_{\m,\n}}e_\n e^{-\zz_{\m,\n}}=e_\m.
$$
The index $[\m:\n]=\tt^{-1}$ is related to the geodesic distance between $e_\n$ and $e_\m$, meaured with the usual norm of $\m_2$, or with the $\rho$-norms ($1\le\rho<\infty$):
\begin{teo}
With the above notations,
$$
d(e_\n,e_\m)=\|\zz_{\m,\n}\|=\arccos(\tt^{1/2}) \ \hbox{ and } \ d_\rho(e_\n,e_\m)=\|\zz_{\m,\n}\|_\rho=\tt^{1/\rho} \arccos(\tt^{1/2}).  
$$
\end{teo}
\begin{proof}
The projections $e_\n$ and $e_\m$ act in $L^2(\m_1,\tau)$. Denote by $e_\n'$ and $e_\m'$ the generic part of these projections, acting on the Hilbet space $\h'\subset L^2(\m_1,\tau)$. By Halmos' theorem, there exists an isometric isomorphism between $\h'$ and $\l\times\l$ such that $e_\n'$ and $e_\m'$ are carried, respectively, onto
$$
P_\n=\left( \begin{array}{cc} 1 & 0 \\ 0 & 0 \end{array} \right) \ \hbox{ and } \ P_\m=\left( \begin{array}{cc} \cos^2(X) & \cos(X)\sin(X) \\ \cos(X)\sin(X) & \sin^2(X) \end{array} \right),
$$
where $0\le X\le \pi/2$. Note that since the only (non trivial) non generic part of $e_\n$ and $e_\m$ is $e_\n^\perp\wedge e_\m^\perp$, on which both $e_\n$ and $e_\m$ act trivially, we have that $e_\n' e_\m' e_\n'=e_\n e_\m e_\n=\tt e_\n$. Therefore,
$$
\tt \left( \begin{array}{cc} 1 & 0 \\ 0 & 0 \end{array} \right)=\tt P_\n=P_\n P_\m P_\n=\left( \begin{array}{cc} \cos^2(X) & \cos(X)\sin(X) \\ \cos(X)\sin(X) & \sin^2(X)
 \end{array} \right) ,
$$ 
i.e., $\cos(X)=\tt^{1/2} 1_\l$, and therefore $X$ is a scalar multiple of the identity in $\l$: $X=\arccos(\tt^{1/2}) 1_\l$. The unique exponent $\zz=\zz_{\m,\n}$ of the geodesic joining $e_\n$ and $e_\m$, is zero on the non generic part $e_\n^\perp\wedge e_\m^\perp$, and in the generic part is related (via the Halmos' isomorphism) to the operator
$$
Z=\left( \begin{array}{cc} 0 & X \\ -X & 0 \end{array} \right).
$$
Then $\zz^*\zz$ corresponds to
$$
Z^*=\left( \begin{array}{cc} X^2 & 0 \\ 0 & X^2 \end{array} \right)=(\arccos(\tt^{1/2}))^2 \left( \begin{array}{cc} 1 & 0 \\ 0 & 1 \end{array} \right).
$$
Therefore $\zz^*\zz=(\arccos(\tt^{1/2}))^2e_\n$. The geodesic distance induced by the usual operator norm is
$$
d(e_\n,e_\m)=\|\zz\|=\|\zz^*\zz\|^{1/2}=\arccos(\tt^{1/2}),
$$
and the one induced by the $\rho$ norm is
$$
d_\rho(e_\n,e_\m)=\|\zz\|_\rho=\arccos(\tt^{1/2})(\tau(e_\n))^{1/\rho}=\tt^{1/\rho}\arccos(\tt^{1/2}).
$$
\end{proof}

Next, we consider the case of two projections arising from two subfactors $\n_0, \n_1\subset \m$. These give rise to two orthogonal projections  $e_0,e_1$ in $\b(L^2(\m,\tau))$. 
We make the assumption that both inclusions have finite index: $[\m:\n_0], [\m:\n_1]<\infty$. 
If both projections lie in the same {\bf II}$_1$ factor $\m_0\supset\m$ (with trace $\tau$ extending the trace of $\m$), then a necessary and sufficient condition for the existence of a geodesic joining $e_0$ and $e_1$ is $\tau(e_0)=\tau(e_1)$.

\begin{lem}
Let $\n_0,\n_1\subset\m$ be finite index subfactors. Then there exists a {\bf II}$_1$ factor $\m_0$ such that $\m\subset\m_0$ has finite index, and $e_0,e_1\in\m_0$.
\end{lem}
\begin{proof}
Let $E_i:\m\to\n_i$, $i=0,1$, be the unique trace preserving conditional expectations, giving rise to the  orthogonal projections $e_0, e_1$. Let $\m_1=<\m, e_0>$, and $F:\m_1\to\m$ the corresponding expectation. Note that $F_1=E_1F:\m_1\to\n_1$ is a conditional expectation, which is trace invariant (for the trace of $\m_1$), and which corresponds to the finite index inclusion $\n_1\subset\m_1$: $[\m_1:\n_1]=[\m_1:\m][\m,\n_1]$. Let $f_1$ be the orthogonal projection in $\b(L^2(\m_1))$ induced by this inclusion, and $\m_0=<\m_1,f_1>$, which is a finite factor with 
$$
[\m_0:\m_1]=[\m_1:\m][\m,\n_1]<\infty.
$$
We claim that $f_1=e_1$. Denote by $[x]$ the element $x\in\m_1$ regarded as a vector in $L^2(\m_1)$. Then, if $x\in\m$,
$$
f_1([x])=[F_1(x)]=[E_1(F(x))]=[E_1(x)]=e_1([x]).
$$
If $\xi\in\m^\perp$, then $f_1(\xi)=e_1(\xi)=0$.
\end{proof}
\begin{rem}
Note that $\m_0\subset \{\m, e_0, e_1\}''\subset \b(L^2(\m_1))$. However, $\m$, $e_0$ and $e_1$ act also on $L^2(\m)$. Thus, the algebra $\{\m, e_0, e_1\}''\subset \b(L^2(\m_1))$ is $*$-isomomorphic (by means of a normal isomorphism, given by restriction to $L^2(\m)$) to the von Neumann {\bf II}$_1$ factor $\m_{1,2}:=<\m,e_0,e_1>\subset\b(L^2(\m))$.
\end{rem}
\begin{prop}
Let $\n_0,\n_1\subset\m$ be finite index subfactors. Then there exists a geodesic joining $e_0$ and $e_1$ if and only if $[\m:\n_0]=[\m:\n_1]$.
\end{prop}
\begin{proof}
$[\m:\n_0]=[\m:\n_1]=\tt^{-1}$ if and only if $\tau_{\m_0}(e_0)=\tau_{\m_0}(e_1)=\tt$.
\end{proof}
With the same notations as above, we have the following:
\begin{lem} 
Suppose that $\|e_0-e_1\|<1$, and let $\delta(t)=e^{tz}e_0e^{-tz}$, $t\in[0,1]$, be the unique geodesic of $\p_{\m_0}$ joining $\delta(0)=e_0$ and $\delta(1)=e_1$. Then 
$$
E_t=\delta_t|_{\m}:\m\to \n_t:=e^{tz}\n_0e^{-tz}\subset\m
$$
is a pointwise smooth path of conditional expectations joining $E_{\n_0}$ and $E_{\n_1}$ (i.e., the map $[0,1]\ni t\mapsto E_t(a)\in\m$ is $C^1$ for all $a\in\m$).
\end{lem}
\begin{proof}
We shall use repeatedly the following argument. Suppose that $m\in\m_0$ is normal, and satisfies that $m[\m]\subset[\m]$, i.e., $m$ as an operator acting in $L^2(\m)$, leaves the dense linear manifold $[\m]=\{[x]: x\in\m\}$ invariant, and let $f$ be a continuous function in the spectrum $\sigma(m)$ of $m$. Then $f(m)$ also leaves $[\m]$ invariant. Indeed, let $p_k(z,\bar{z})$ be polynomials in $z$ and $\bar{z}$ which converge uniformly to $f(z)$ in $\sigma(m)$. Clearly $p_k(m,m^*)$ leave $[\m]$ invariant.  Let $x\in\m$. Then, if we denote by $L_x$ the element $x$ acting by left multiplication on $L^2(\m)$,
$$
\|p_k(m,m^*)(x)-p_j(m,m^*)(x)\|_{\m}=\|(p_k(m,m^*)-p_j(m,m^*))L_x\|_{\b(L^2(\m))}
$$
$$
\le \|p_k(m,m^*)-p_j(m,m^*)\|_{\b(L^2(\m))}\|L_x\|_{\b(L^2(\m))}=
\|p_k(m,m^*)-p_j(m,m^*)\|_{\b(L^2(\m))}\|x\|.
$$
It follows that $p_k(m,m^*)(x)$ is a Cauchy sequence in $\m$, which converges to $f(m)(x)\in\m$. 

Consider now the element $e_0+e_1-1\in\m_0$. Note  that $\|e_0-e_1\|<1$ implies that $e_0+e_1-1$ is invertible. Clearly,  $e_0+e_1-1$ leaves $[\m]$ invariant:
$$
(e_0 +e_1-1)([x])=[E_0(x)+E_1(x)-1]\in[\m]
$$
for all $x\in\m$. By the above argument, it follows that $|e_0+e_1-1|^{-1}$ leaves $[\m]$ invariant. Thus
$$
e^z=(2e_0-1)(e_0+e_1-1)|e_0+e_1-1|^{-1}
$$
leaves $[\m]$ invariant. On the other hand, as remarked before, the fact that $\|e_0-e_1\|<1$ also implies that $\|e^z-1\|<\sqrt2<2$ (or equivalently, that $\|z\|<\pi/2$). It follows that there is a continuous logarithm defined in the spectrum of $e^z$,  $arg:\sigma(e^z)\to(-\pi/2,\pi/2)$. Therefore, again using the  argument at the beginning of this proof, it follows that $z$ leaves $[\m]$ invariant. Therefore, $e^{tz}$ leave $[\m]$ invariant for $t\in[0,1]$. It follows that $\delta(t)$, restricted to $\m$, induce the linear mappings
$$
\delta(t)|_{\m}=e^{tz}e_0e{-tz}|_\m:\m\to\m.
$$
The range of $\delta(t)|_\m$ is $e^{tz}L^2(\n_0) e^{-tz}\cap \m=e^{tz}\n_0 e^{-tz}=\n_t$. Clearly these maps are idempotents, $*$-preserving, normal, and contractive for the norm of $\m$. Thus, by the theorem of Tomiyama \cite{tomiyama}, they are normal conditional expectations, interpolating between $E_0$ and $E_1$. The fact that the path is strongly smooth is also clear.  
\end{proof}
\begin{rem}
Let us recall Theorem 2.6 of \cite{strongly smooth}:

Let $\a$ be a unital C$^*$-algebra and suppose that for $t\in[0,1]$ one has subalgebras $1\in \b_t\subset\a$  and conditional expectations $E_t : \a\to\b_t$ . Assume that  for each $a\in\a$, the map $t\mapsto E_t (a)\in\a$ is continuously differentiable. Denote by $dE_t:\a\to\a$ the derivative of $E_t$: $dEt (a) = \frac{d}{dt} E_t (a)$. For each fixed $t$, the operator $dE_t : \a\to\a$ is bounded. 
Consider the differential equation, for $a\in\a$,  
\begin{equation}\label{ec transporte}
\left\{ \begin{array}{l} \dot{\alpha}(t)=[dE_t.E_t](\alpha(t)) \\ \alpha(0)=a \end{array} \right. .
\end{equation}
We call this equation the {\it parallel transport equation}. Denote by $\Gamma_t$ the propagator of this equation, i.e., the map $\Gamma_t:\a\to\a$ given by the solutions: $\Gamma_t(a)=\alpha(t)$ with $\alpha(0)=a$.
Then
\begin{itemize}
\item
$\Gamma_tE_0\Gamma_{-t}=E_t$, and
\item
$\Gamma_t|_{\b_0}:\b_0\to\b_t$ is a C$^*$-algebra isomorphism.
\end{itemize}
\end{rem}
\begin{coro}
Let $\n_0,\n_1\subset\m$ be subfactors with $\|e_0-e_1\|<1$. Then the exponentials $\Gamma_t=e^{tz}$ which induce the unique geodesic $\delta(t)=e^{tz}e_0e^{-tz}$ joining $e_0$ and $e_1$ in $\p_{\m_0}$  ($\m_0=<\m,e_0,e_1>''$), satisfy that 
$$
\Gamma_t|_{\n_0}:\n_0\to\n_t=e^{tz}\n_0e^{-tz}
$$
are normal $*$-isomorphisms. In particular, $\n_0$ and $\n_1$ are isomorphic. 
\end{coro}
\begin{proof}
It is straightforward to verify that $e^{tz}$ are the propagators  $\Gamma_t$ of equation (\ref{ec transporte}) in this case: since $E_t=\delta(t)|_{\m}=e^{tz}|_\m E_0e^{-tz}|_\m$, we have that (the maps below are restricted to $\m$):
$$
dE_t=ze^{tz}E_0e^{-tz}-e^{tz}E_0e^{-tz}z
$$
and after straightforward computations
$$
[dE_t,E_t]=ze^{tz}E_0e^{-tz}-2e^{tz}E_0zE_0e^{-tz}+e^{tz}E_0e^{-tz}z.
$$
Since $z$ is $e_0$ co-diagonal, it maps $L^2(\n_0)$ into $L^2(\n_0)^\perp$, and therefore $E_0zE_0=0$. Thus, 
$$
[dE_t,E_t]e^{tz}=ze^{tz}E_0+e^{tz}E_0z=e^{tz}(zE_0+E_0z).
$$
The fact that the element $z$ is $e_0$ co-diagonal, also means that 
$$
z=e_0z(1-e_0)+(1-e_0)ze_0=e_0z-2e_0ze_0+ze_0=ze_0+e_0z.
$$
Restricted to $\m$ gives $zE_0+E_0z=z$. Therefore, for $x\in\m$,
$$
[dE_t,E_t]e^{tz}x=e^{tz}(zE_0+E_0z)zx=e^{tz}zx=(e^{tz}x)^{\cdot}.
$$
That is, $\alpha(t)=e^{tz}x$ is the solution of (\ref{ec transporte}) with $\alpha(0)=x$, i.e. $\Gamma_t=e^{tz}|_\m$ is the propagator of this equation, and the proof follows using Theorem 2.6 of \cite{strongly smooth}
\end{proof}

\begin{rem}
Suppose that  $e_0, e_1$ as above satisfy the  condition $e_0\wedge e_1^\perp=0=e_0^\perp\wedge e_1$ (weaker that $\|e_0-e_1\|<1$), then there exists a unique geodesic $\delta(t)=e^{tz}e_0e^{-tz}$ joining $e_0$ and $e_1$. We would like to know if also in this case, the propagators $\Gamma_t$ of the parallel transport equation induce as in the above case, a curve of automorphisms. Following the same argument as above, it amounts to knowing if the projections $\delta(t)$ induce conditional expectations onto the intermediate algebras $e^{tz}\n_0e^{-tz}$. 
\end{rem}

{\sc (Esteban Andruchow)} {Instituto de Ciencias,  Universidad Nacional de Gral. Sar\-miento,
J.M. Gutierrez 1150,  (1613) Los Polvorines, Argentina and Instituto Argentino de Matem\'atica, `Alberto P. Calder\'on', CONICET, Saavedra 15 3er. piso,
(1083) Buenos Aires, Argentina.}

\end{document}